\documentclass[12pt,a4paper,english]{article}
\usepackage[T1]{fontenc}
\usepackage[latin9]{inputenc}
\usepackage{textcomp}
\usepackage{mathtools}
\usepackage{amsmath}
\usepackage{amsthm}
\usepackage{amssymb}
\usepackage{setspace}
\makeatletter

\pdfpageheight\paperheight
\pdfpagewidth\paperwidth

\theoremstyle{definition}
\newtheorem{defn}{\protect\definitionname}
\theoremstyle{remark}
\newtheorem{rem}{\protect\remarkname}
\theoremstyle{plain}
\newtheorem{thm}{\protect\theoremname}

\@ifundefined{date}{}{\date{}}
\makeatother

\usepackage{babel}
\providecommand{\definitionname}{Definition}
\providecommand{\remarkname}{Remark}
\providecommand{\theoremname}{Theorem}

\begin{document}
\title{A note on the semiclassical measure at singular points of the boundary
of the Bunimovich stadium}
\author{Dan Mangoubi and Adi Weller Weiser}
\maketitle
\begin{abstract}
An argument by Hassell proving the existence of a Bunimovich stadium
for which there are semiclassical measures giving positive mass to
the submanifold of bouncing ball trajectories uses a notion of non-gliding
points. However, this notion is defined only for domains with~$C^{2}$-boundaries.
The purpose of this note is to clarify the argument.
\end{abstract}

\section{Introduction}

In a celebrated paper \cite{hassell2010ergodic} Hassell proves the
existence of quantum ergodic manifolds which are not quantum uniquely
ergodic. Furthermore, Hassell shows that there is a Bunimovich stadium,
$M$, with a semiclassical measure giving positive mass to the submanifold
of bouncing ball trajectories. For this refinement, in order to rule
out formations of certain non-uniform semiclassical measures, Hassell
applies a theorem of Burq and Gérard \cite{burq-gerard} showing that
a boundary semiclassical measure cannot put any mass on the set of
non-gliding points, a special subset of $S^{*}\partial M$ where the
classical trajectories passing through are not affected by the boundary
(see Appendix). However, the definition of non-gliding points requires
at least~$C^{2}$-smoothness of the boundary, while the boundary
of the Bunimovich stadium is only~$C^{1,1}$-smooth. Moreover, at
a neighborhood of a singular point~$p$ of the boundary, one has
points in $S^{*}\partial M$ which are far from being non-gliding,
since the curvature jumps there from zero to a positive constant,
and it is not a-priori clear from Burq and Gérard's argument whether
one should regard the points of~$S_{p}^{*}\partial M$ as gliding
or non-gliding. The aim of this note is to clarify Hassell's argument,
by showing (in Theorem \ref{The Theorem}) that indeed the semiclassical
measure on the boundary gives zero mass to the portion of~$S^{*}\partial M$
lying above the \emph{closure} of the straight part of~$\partial M$.
Our result is an instance where one can relax the smoothness assumption
in~\cite{burq-gerard} and hints that in two dimensions the~$C^{1,1}$
assumption may be sufficient to conclude that the semiclassical measure
of the boundary vanishes on the closure of the subset of~$S^{*}\partial M$
where the classical trajectories passing through are not reflected
as a replacement for non-gliding points in the smooth case. In this
paper, we treat the Bunimovich stadium case. The proof, like in \cite{burq-gerard},
uses Gérard-Leichtnam's transport equation from \cite{gerard-leichtnam},
recalled here in Section \ref{sec:Background}, and follows similar
lines of microlocalization on non-gliding points. However, unlike
in \cite{burq-gerard}, we do not make any change of variables, allowing
us to treat the singular points of the boundary as well.

\section*{Acknowledgments}

This paper arose from a discussion with Patrick Gérard. Patrick showed
us an initial ad hoc argument explaining the application of~\cite{burq-gerard}
in~\cite{hassell2010ergodic} (see §\ref{subsec:An-argument-of-Gerard}).
We are very grateful to Nicolas Burq, Patrick Gérard and Stéphane
Nonnenmacher for their encouragement to write this note and for many
hours of discussions on topics of Quantum Chaos. D.\ M.\ especially
thanks Stéphane for the invitation to spend a sabbatical leave in
Paris-Sud XI in a very pleasant research atmosphere, during which
this work started. A.\ W.\ appreciates her welcome in the course
of a visit to Paris-Sud XI. Both authors thank Elon Lindenstrauss
for his guidance. The financial supports of LabEx Mathématiques Hadamard
and the French government are gratefully acknowledged. This paper
is part of A.\ W.'s research towards a Ph.D.\ dissertation, conducted
at the Hebrew University of Jerusalem. This research was partially
supported by ISF Grant No.\ 681/18.

\section{Background\label{sec:Background}: Gérard-Leichtnam's transport equation }

By Egorov's theorem, in the case of a closed manifold $M$ a semiclassical
measure $\mu$ is invariant under the geodesic flow; or equivalently
$\xi\cdot\partial_{x}\mu=0$, where~$\left(x,\xi\right)$ are canonical
coordinates on $T^{*}M$. In the presence of a~$C^{1,1}$-boundary
a semiclassical measure $\mu$ evolves according to Gérard-Leichtnam's
equation. It expresses the distribution $\xi\cdot\partial_{x}\mu$
in terms of a corresponding semiclassical measure $\nu$ on $T^{*}\partial M$,
and Snell's law is manifested in the symmetry of this equation. Here
we explain the notions of a corresponding semiclassical measure on~$T^{*}\partial M$,
hyperbolic point and glancing point and we are then able to present
the equation.

Let $M\subset\mathbb{R}^{2}$ be a bounded domain with $C^{1,1}$-boundary,
and $\{u_{j}\}$ an orthonormal basis of Laplace eigenfunctions of
$M$ with Dirichlet's boundary conditions and with increasing positive
eigenvalues $\{E_{j}\}$ respectively. For the definition below we
consider $\{u_{j}\}$ as a sequence in $L^{2}(\mathbb{R}^{2})$.
\begin{defn}
$M$-Semiclassical measures on $T^{*}\mathbb{R}^{2}$.\label{def:Semiclassical-measures}

Let $\{u_{k_{j}}\}$ be a subsequence such that for all $a\in C_{c}^{\infty}\left(T^{*}\mathbb{R}^{2}\right)$
there exists the limit 
\[
\lim_{j\rightarrow\infty}\left\langle Op_{h_{k_{j}}}\left(a\right)u_{k_{j}},u_{k_{j}}\right\rangle 
\]
where $h_{k_{j}}=E_{k_{j}}^{-1/2}$ and $Op_{h}\left(a\right)u\left(x\right)=\frac{1}{\left(2\pi\right)^{2}}\int a\left(x,h\xi\right)e^{ix\cdot\xi}\hat{u}\left(\xi\right)d\xi$
is the standard quantization. Then we define the distribution $\mu$
by 
\[
\left\langle \mu,a\right\rangle =\lim_{j\rightarrow\infty}\left\langle Op_{h_{k_{j}}}\left(a\right)u_{k_{j}},u_{k_{j}}\right\rangle \:.
\]
We call $\mu$ an \emph{$M$-semiclassical measure} associated to
the subsequence $\{u_{k_{j}}\}$. 
\end{defn}
\begin{rem}
The distribution in the above definition is actually a positive measure
supported on $\overline{T^{*}M}$ (see \cite[pp. 565]{gerard-leichtnam}).
\end{rem}
To define the corresponding boundary semiclassical measures recall
that in \cite{gerard-leichtnam} Gérard and Leichtnam showed that
the standard quantization procedure may be extended to the following
class of symbols defined on any closed~$C^{1}$-manifold $N$
\begin{multline*}
\Sigma_{c}\left(T^{*}N\right)=\left\{ b\in C_{c}\left(T^{*}N\right)\right|\\
\left.b\text{ has continuous vertical derivatives up to order }n+1\right\} 
\end{multline*}
where $n$ is the dimension of $N$, while keeping the resulting operators~$Op_{h}\left(b\right)$
bounded on $L^{2}\left(N\right)$. 

In our case, this extension allows to define below a notion of semiclassical
measure on $T^{*}\partial M$. For the sake of this definition we
also recall that due to our assumption that $\partial M$ is~$C^{1,1}$-smooth,
the normal derivatives on the boundary $\partial_{n}u_{k}$ exist
in $L^{2}\left(\partial M\right)$, and moreover, the sequence $h_{k}\partial_{n}u_{k}$
is bounded in~$L^{2}\left(\partial M\right)$ (see e.g.\ \cite[Lemma 2.1]{gerard-leichtnam}).
\begin{defn}
Corresponding semiclassical measures on $T^{*}\partial M$ for a Dirichlet
problem.\label{def:Semiclassical-measures-on-boundary}

For an $M$-semiclassical measure $\mu$ associated to the subsequence
$\{u_{k_{j}}\}$, assume that for all $b\in\Sigma_{c}\left(T^{*}\partial M\right)$
there exists the limit 
\[
\lim_{j\rightarrow\infty}\left\langle Op_{h_{k_{j}}}\left(b\right)h_{k_{j}}\partial_{n}u_{k_{j}},h_{k_{j}}\partial_{n}u_{k_{j}}\right\rangle 
\]
Then, define a positive measure $\nu$ on $T^{*}\partial M$ by 
\[
\left\langle \nu,b\right\rangle =\lim_{j\rightarrow\infty}\left\langle Op_{h_{k_{j}}}\left(b\right)h_{k_{j}}\partial_{n}u_{k_{j}},h_{k_{j}}\partial_{n}u_{k_{j}}\right\rangle \:.
\]
We call $\nu$ a \emph{boundary semiclassical measure corresponding
to $\mu$}. 
\end{defn}
\begin{rem}
For any given subsequence $\{u_{k_{j}}\}$ there is a subsequence
$\{u_{k_{j_{l}}}\}$ for which the limits in Definitions \ref{def:Semiclassical-measures}
and \ref{def:Semiclassical-measures-on-boundary} exist (see \cite[pp. 565, 576]{gerard-leichtnam}).
\end{rem}
Next, points in $T^{*}\partial M$ are classified according to their
dynamical behaviour \cite{Melrose-1975} (see also \cite{Melrose-gliding}).
The definitions below are only for the case of the billiard dynamics.
\begin{defn}
\label{def:Glancing-points}Glancing points
\[
\mathcal{G}=\left\{ \left(x,\xi\right)\in T^{*}\partial M\:|\:\left|\xi\right|=1\right\} =S^{*}\partial M
\]
\end{defn}
\begin{defn}
Hyperbolic points
\[
\mathcal{H}=\left\{ \left(x,\xi\right)\in T^{*}\partial M\:|\:\left|\xi\right|<1\right\} 
\]
\end{defn}
The Gérard-Leichtnam transport equation is
\begin{thm}[\cite{gerard-leichtnam}]
For $\nu$ a boundary semiclassical measure corresponding to an $M$-semiclassical
measure $\mu$ we have for all $a\in C_{c}^{\infty}\left(T^{*}\mathbb{R}^{2}\right)$
\begin{equation}
-\int_{T^{*}\mathbb{R}^{2}}\xi\cdot\partial_{x}a\,d\mu=\int_{\mathcal{H}\cup\mathcal{G}}\frac{a\left(x\left(\rho\right),\xi^{+}\left(\rho\right)\right)-a\left(x\left(\rho\right),\xi^{-}\left(\rho\right)\right)}{\left\langle \xi^{+}\left(\rho\right)-\xi^{-}\left(\rho\right),n\left(x\left(\rho\right)\right)\right\rangle }d\nu\left(\rho\right)\label{eq:gerard-leichtnam}
\end{equation}
where $n\left(x\left(\rho\right)\right)$ is the inward pointing normal
at $x\left(\rho\right)$, and $\xi^{+}\left(\rho\right)$ is the co-vector
in $S_{x\left(\rho\right)}^{*}\mathbb{R}^{2}$ such that its orthogonal
projection on $T_{x\left(\rho\right)}^{*}\partial M$ is equal to
$\rho$'s co-vector and so that $\left\langle \xi^{+}\left(\rho\right),n\left(x\left(\rho\right)\right)\right\rangle \geq0$.
The co-vector $\xi^{-}\left(\rho\right)$ is similarly defined with
$\left\langle \xi^{-}\left(\rho\right),n\left(x\left(\rho\right)\right)\right\rangle \leq0$.
\end{thm}

\section{Zero mass on the set of glancing points of the straight part of the
boundary}

It is shown in \cite{burq-gerard} that in the case that the boundary
is smooth the set of non-gliding points is of boundary semiclassical
measure zero. We extend the proof to the case of the Bunimovich stadium
and show that the set of glancing points lying above the closure of
the straight part of the boundary is of boundary semiclassical measure
zero. The main new point is the avoidance of change of variables which
allows us to handle non-smooth points of $\partial M$.

Let $M=\left(\left[-a,a\right]\times\left(0,1\right)\right)\cup\left\{ \left(x\pm a\right)^{2}+\left(y-\frac{1}{2}\right)^{2}<\frac{1}{4}\right\} $.
Decompose $\partial M$ as $S\cup C$ where $S$ is the straight part
of the boundary and $C$ is its curved part. To be precise, we let
$S=S_{0}\cup S_{1}$ where $S_{i}=\left[-a,a\right]\times\left\{ i\right\} $
and put $C=\left(\partial M\right)\backslash S$. Let $\pi:T^{*}\partial M\rightarrow\partial M$
and $\pi_{\mathbb{R}^{2}}:T^{*}\mathbb{R}^{2}\rightarrow\mathbb{R}^{2}$
denote the projection maps.
\begin{thm}
\label{The Theorem}Let $\nu$ be a semiclassical measure on~$T^{*}\partial M$
corresponding to an $M$-semiclassical measure. Let $\text{\ensuremath{\mathcal{G}}}$
be the set of glancing points of~$T^{*}\partial M$. Then $\nu\left(\ensuremath{\mathcal{G}}\cap\pi^{-1}\left(S\right)\right)=0$.
\end{thm}
\begin{proof}
Let $\mu$ be an~$M$-semiclassical measure to which $\nu$ corresponds.
Fix a point~$\rho_{0}\in\mathcal{G}\cap\pi^{-1}\left(S_{0}\right)$
and consider 
\[
a_{\varepsilon}\left(x,\xi\right)=\xi_{2}b\left(x_{1},\frac{x_{2}}{\varepsilon},\frac{\xi_{2}^{2}}{\varepsilon}\right)\varphi\left(\xi_{1}\right)
\]
where $b\in C_{c}^{\infty}\left(\mathbb{R}^{3}\right)$ is nonnegative
with $b\left(x_{1}\left(\rho_{0}\right),0,0\right)>0$ and where the
cutoff function~$\varphi$ is in~$C_{c}^{\infty}\left(\mathbb{R}\right)$
with $\left|\varphi\right|\leq1$ and $\varphi\left(\xi_{1}\right)=1$
for $\left|\xi_{1}\right|\leq1$. In equation~(\ref{eq:gerard-leichtnam})
we make the substitution $a=a_{\varepsilon}$ and take the limit as
$\varepsilon$ tends to zero.

For calculating the left hand side of (\ref{eq:gerard-leichtnam})
we note that
\[
\xi\cdot\partial_{x}a_{\varepsilon}=\xi_{1}\xi_{2}\varphi\left(\xi_{1}\right)\left(\partial_{1}b\right)\left(x_{1},\frac{x_{2}}{\varepsilon},\frac{\xi_{2}^{2}}{\varepsilon}\right)+\xi_{2}^{2}\varphi\left(\xi_{1}\right)\frac{1}{\varepsilon}\left(\partial_{2}b\right)\left(x_{1},\frac{x_{2}}{\varepsilon},\frac{\xi_{2}^{2}}{\varepsilon}\right)
\]
We claim that $\xi\cdot\partial_{x}a_{\varepsilon}\longrightarrow0$
as $\varepsilon\rightarrow0$. Indeed, if $\xi_{2}=0$ then $\xi\cdot\partial_{x}a_{\varepsilon}=0$.
If~$\xi_{2}\neq0$, then for all $\varepsilon$ small enough we have
that $\partial_{i}b\left(x_{1},\frac{x_{2}}{\varepsilon},\frac{\xi_{2}^{2}}{\varepsilon}\right)\equiv0$
since~$b$ is compactly supported. Furthermore, $\xi\cdot\partial_{x}a_{\varepsilon}$
is dominated by a constant independent of $\varepsilon\leq1$. In
fact, we have
\begin{multline*}
\left|\xi_{1}\xi_{2}\varphi\left(\xi_{1}\right)\left(\partial_{1}b\right)\left(x_{1},\frac{x_{2}}{\varepsilon},\frac{\xi_{2}^{2}}{\varepsilon}\right)\right|=\left|\xi_{1}\varphi\left(\xi_{1}\right)\frac{\xi_{2}}{\sqrt{\varepsilon}}\left(\partial_{1}b\right)\left(x_{1},\frac{x_{2}}{\varepsilon},\frac{\xi_{2}^{2}}{\varepsilon}\right)\sqrt{\varepsilon}\right|\\
\leq\sup_{x'_{1},x'_{2},y,\xi'_{1}}\left|y\partial_{1}b\left(x'_{1},x'_{2},y^{2}\right)\right|\left|\xi'_{1}\varphi\left(\xi'_{1}\right)\right|
\end{multline*}
while 
\[
\left|\frac{\xi_{2}^{2}}{\varepsilon}\varphi\left(\xi_{1}\right)\left(\partial_{2}b\right)\left(x_{1},\frac{x_{2}}{\varepsilon},\frac{\xi_{2}^{2}}{\varepsilon}\right)\right|\leq\sup_{x'_{1},x'_{2},y}\left|y\partial_{2}b\left(x'_{1},x'_{2},y\right)\right|\;.
\]
Since $\mu$ has compact support ($\text{supp}\mu\subseteq S^{*}\mathbb{R}^{2}\cap\pi_{\mathbb{R}^{2}}^{-1}\left(\overline{M}\right)$,
see \cite{gerard-leichtnam}) we may conclude by the Lebesgue Dominated
Convergence Theorem that
\begin{equation}
\int_{T^{*}\mathbb{R}^{2}}\xi\cdot\partial_{x}a_{\varepsilon}\,d\mu\underset{\varepsilon\rightarrow0}{\longrightarrow}0\;.\label{eq:LHS}
\end{equation}

Next, we calculate the right hand side of (\ref{eq:gerard-leichtnam}).
We set 
\[
b_{\varepsilon}\left(x_{1},x_{2},\xi_{2}\right)\coloneqq b\left(x_{1},\frac{x_{2}}{\varepsilon},\frac{\xi_{2}^{2}}{\varepsilon}\right)\;.
\]
For $y\in\partial M$, let $n^{*}\left(y\right)\in S_{y}^{*}\mathbb{R}^{2}$
be the co-vector such that $\left\langle n^{*}\left(y\right),n\left(y\right)\right\rangle =1$
and $\left\langle n^{*}\left(y\right),\dot{\gamma}\right\rangle =0$
where $\dot{\gamma}\in S_{y}\partial M$. Note that for $\rho\in\mathcal{H}\cup\mathcal{G}$
\[
\xi^{\pm}\left(\rho\right)=\xi\left(\rho\right)\pm\sqrt{1-\left|\xi\left(\rho\right)\right|^{2}}n^{*}\left(x\left(\rho\right)\right)\;.
\]
It will be convenient to write 
\[
\xi^{\pm}\left(\rho\right)=\xi_{1}^{\pm}\left(\rho\right)dx^{1}+\xi_{2}^{\pm}\left(\rho\right)dx^{2}\:.
\]
The numerator of the integrand in the right hand side of (\ref{eq:gerard-leichtnam})
is
\begin{multline*}
a_{\varepsilon}\left(x\left(\rho\right),\xi^{+}\left(\rho\right)\right)-a_{\varepsilon}\left(x\left(\rho\right),\xi^{-}\left(\rho\right)\right)\\
\overset{\left|\xi_{1}^{\pm}\right|\leq1}{=}\xi_{2}^{+}\left(\rho\right)b_{\varepsilon}\left(x\left(\rho\right),\xi_{2}^{+}\left(\rho\right)\right)-\xi_{2}^{-}\left(\rho\right)b_{\varepsilon}\left(x\left(\rho\right),\xi_{2}^{-}\left(\rho\right)\right)\\
=f_{\varepsilon}\left(x\left(\rho\right),\xi_{2}^{+}\left(\rho\right)\right)-f_{\varepsilon}\left(x\left(\rho\right),\xi_{2}^{-}\left(\rho\right)\right)
\end{multline*}
where $f_{\varepsilon}\left(x,\xi_{2}\right):=\xi_{2}b_{\varepsilon}\left(x,\xi_{2}\right)$.
Hence we can write the integrand in the right hand side of (\ref{eq:gerard-leichtnam})
as
\begin{multline}
A_{\varepsilon}\left(\rho\right)=\frac{f_{\varepsilon}\left(x\left(\rho\right),\xi_{2}^{+}\left(\rho\right)\right)-f_{\varepsilon}\left(x\left(\rho\right),\xi_{2}^{-}\left(\rho\right)\right)}{\|\xi^{+}\left(\rho\right)-\xi^{-}\left(\rho\right)\|}\\
=\frac{\xi_{2}^{+}\left(\rho\right)-\xi_{2}^{-}\left(\rho\right)}{\|\xi^{+}\left(\rho\right)-\xi^{-}\left(\rho\right)\|}\cdot\frac{1}{\xi_{2}^{+}\left(\rho\right)-\xi_{2}^{-}\left(\rho\right)}\int_{\xi_{2}^{-}\left(\rho\right)}^{\xi_{2}^{+}\left(\rho\right)}\partial_{\xi_{2}}f_{\varepsilon}\left(x\left(\rho\right),y\right)dy\\
=\frac{\xi_{2}^{+}\left(\rho\right)-\xi_{2}^{-}\left(\rho\right)}{\|\xi^{+}\left(\rho\right)-\xi^{-}\left(\rho\right)\|}\int_{0}^{1}\partial_{\xi_{2}}f_{\varepsilon}\left(x\left(\rho\right),t\xi_{2}^{+}\left(\rho\right)+\left(1-t\right)\xi_{2}^{-}\left(\rho\right)\right)dt\:.\label{eq:Aepsilon-as-integral}
\end{multline}
To take the limit notice first that
\begin{multline}
\left|\partial_{\xi_{2}}f_{\varepsilon}\left(x_{1},x_{2},\xi_{2}\right)\right|=\left|b_{\varepsilon}\left(x_{1},x_{2},\xi_{2}\right)+2\frac{\xi_{2}^{2}}{\varepsilon}\left(\partial_{3}b\right)\left(x_{1},\frac{x_{2}}{\varepsilon},\frac{\xi_{2}^{2}}{\varepsilon}\right)\right|\\
\le\sup_{y}\left|b\left(x_{1},\frac{x_{2}}{\varepsilon},y\right)\right|+2\sup_{y}\left|y\partial_{3}b\left(x_{1},\frac{x_{2}}{\varepsilon},y\right)\right|\label{eq:Aepsilon-bound}
\end{multline}
and 
\begin{equation}
\frac{\left|\xi_{2}^{+}\left(\rho\right)-\xi_{2}^{-}\left(\rho\right)\right|}{\|\xi^{+}\left(\rho\right)-\xi^{-}\left(\rho\right)\|}\leq1\;.\label{eq:Aepsilon-bound-2}
\end{equation}
As $x_{2}\left(\rho\right)\neq0$ when $\rho\in\pi^{-1}\left(C\cup S_{1}\right)$
and $b$ is with compact support we can infer that $\lim_{\varepsilon\rightarrow0}A_{\varepsilon}\left(\rho\right)=0$
in this case. For $\rho\in\pi^{-1}\left(S_{0}\right)$ notice that
we have $\xi_{2}^{+}\left(\rho\right)=-\xi_{2}^{-}\left(\rho\right)=\sqrt{1-\left|\xi\left(\rho\right)\right|^{2}}$,
and since $b_{\varepsilon}\left(x,\xi_{2}\right)=b_{\varepsilon}\left(x,-\xi_{2}\right)$
we get
\begin{multline*}
A_{\varepsilon}\left(\rho\right)=\frac{\xi_{2}^{+}\left(\rho\right)b_{\varepsilon}\left(x\left(\rho\right),\xi_{2}^{+}\left(\rho\right)\right)-\xi_{2}^{-}\left(\rho\right)b_{\varepsilon}\left(x\left(\rho\right),\xi_{2}^{-}\left(\rho\right)\right)}{\xi_{2}^{+}\left(\rho\right)-\xi_{2}^{-}\left(\rho\right)}\\
=b_{\varepsilon}\left(x_{1}\left(\rho\right),0,\sqrt{1-\left|\xi\left(\rho\right)\right|^{2}}\right)=b\left(x_{1}\left(\rho\right),0,\frac{1-\left|\xi\left(\rho\right)\right|^{2}}{\varepsilon}\right)\:.
\end{multline*}
We conclude that
\begin{equation}
A_{\varepsilon}\left(\rho\right)\underset{\varepsilon\rightarrow0}{\longrightarrow}\begin{cases}
b\left(x_{1}\left(\rho\right),0,0\right), & \text{if }\pi\left(\rho\right)\in S_{0}\text{ and }\text{\ensuremath{\left|\xi\left(\rho\right)\right|}}=1,\\
0, & \text{otherwise}.
\end{cases}\label{eq:Aepsilon_limit}
\end{equation}
We verify that $A_{\varepsilon}$ is bounded independently of $\varepsilon$.
Indeed, from the integral expression (\ref{eq:Aepsilon-as-integral})
and the bounds (\ref{eq:Aepsilon-bound}) and (\ref{eq:Aepsilon-bound-2})
it follows that
\[
\left|A_{\varepsilon}\left(\rho\right)\right|\leq\sup\left|b\right|+2\sup_{x_{1}',x_{2}',y}\left|y\partial_{3}b\left(x_{1}',x_{2}',y\right)\right|\;.
\]

From (\ref{eq:Aepsilon_limit}) we have by the Lebesgue Dominated
Convergence Theorem the convergence of the integral
\begin{equation}
\int_{\mathcal{H}\cup\mathcal{G}}A_{\varepsilon}\left(\rho\right)d\nu\left(\rho\right)\underset{\varepsilon\rightarrow0}{\longrightarrow}\int_{\mathcal{G}\cap\pi^{-1}\left(S_{0}\right)}b\left(x_{1}\left(\rho\right),0,0\right)d\nu\left(\rho\right)\;.\label{eq:RHS}
\end{equation}

Comparing (\ref{eq:LHS}) and (\ref{eq:RHS}) we learn that
\[
0=\int_{\mathcal{G}\cap\pi^{-1}\left(S_{0}\right)}b\left(x_{1}\left(\rho\right),0,0\right)d\nu\left(\rho\right)\;.
\]
Since $b\left(x_{1}\left(\rho\right),0,0\right)\geq0$ and $b\left(x_{1}\left(\rho_{0}\right),0,0\right)>0$
we have that $\rho_{0}\notin\text{supp}\left(\nu\right)$. Because
$\rho_{0}\in\mathcal{G}\cap\pi^{-1}\left(S_{0}\right)$ is arbitrary,
we see that $\nu\left(\mathcal{G}\cap\pi^{-1}\left(S_{0}\right)\right)=0$.
Similarly $\nu\left(\mathcal{G}\cap\pi^{-1}\left(S_{1}\right)\right)=0$.
\end{proof}

\subsection{An argument under a vanishing assumption on the curved part \label{subsec:An-argument-of-Gerard}}

In~\cite{hassell2010ergodic} the zero mass of non-gliding points
theorem from~\cite{burq-gerard} is applied in a case where it is
known that the boundary semicalssical measure~$\nu$ on~$T^{*}\partial M$
vanishes on the portion lying above the curved part of~$\partial M$,
$\pi^{-1}\left(C\right)$. In this circumstance one does not need
the full power of Theorem~\ref{The Theorem} and we bring here an
ad hoc approach which was explained to us by Gérard.

Since $\nu$ vanishes on $\pi^{-1}\left(C\right)$, it follows from
Gérard-Leichtnam's equation that if ~$\nu$ corresponds to an $M$-semiclassical
measure~$\mu$, then~$\mu$ is supported on the portion of~$T^{*}\mathbb{R}^{2}$
lying above the closed rectangular part of the billiard table. As
a result, one may replace the singular billiard table by an infinite
strip, and then apply the analysis in~\cite{burq-gerard} on a smooth
domain in order to conclude that~$\nu$ vanishes on~$\mathcal{G}\cap\pi^{-1}(S)$.

\section{Appendix: Non-Gliding points in the case of a two dimensional billiard
with $C^{2}$-boundary}

We recall the definition of a special subset $\mathcal{G}_{\text{ng}}\subset\mathcal{G}\subset T^{*}\partial M$
which is known as the set of non-gliding points. For a two dimensional
$C^{2}$-billiard $M$, one can first positively orient $\partial M$,
and then define
\begin{defn}
Non-gliding points for $M\subset\mathbb{R}^{2}$ with billiard dynamics.
\[
\mathcal{G_{\text{ng}}}=\left\{ \left(x,\xi\right)\in\mathcal{G}\:|\text{ The curvature of \ensuremath{\partial M} at \ensuremath{x} is nonpositive}\right\} 
\]
\end{defn}
\begin{rem}
In \cite{burq-gerard} these points are called non-strictly-gliding,
a term which we avoid due to grammatical ambiguity. 
\end{rem}
More generally \cite{Melrose-1975}, the notion of non-gliding points
is defined in the context of a Hamiltonian dynamical system on a manifold
with boundary~$M$, and in any case it requires that the normal vector
to $\partial M$ be $C^{1}$, or, equivalently, that the boundary
be~$C^{2}$. The set of non-gliding points is the union of the set
of diffractive points and the set of high order glancing points as
defined in \cite{Melrose-1975}. In the case of two dimensional billiards
the definition simplifies to the one above.

\bibliographystyle{abbrv}
\bibliography{refs}

\textsc{Einstein Institute of Mathematics, Edmond J. Safra Campus,
The Hebrew University of Jerusalem, Jerusalem 9190401, Israel}

\emph{Email address: }\texttt{\textbf{dan.mangoubi@mail.huji.ac.il}}

\vspace{3ex}

\textsc{Einstein Institute of Mathematics, Edmond J. Safra Campus,
The Hebrew University of Jerusalem, Jerusalem 9190401, Israel}

\emph{Email address: }\texttt{\textbf{adi.weller@mail.huji.ac.il}}
\end{document}